\documentclass{amsart}

\usepackage{amsmath, amssymb, hyperref, mathrsfs, paralist, thmtools, tikz}
\usetikzlibrary{matrix, arrows}

\definecolor{my-linkcolor}{rgb}{0.75,0,0}
\definecolor{my-citecolor}{rgb}{0,0.5,0}
\definecolor{my-urlcolor}{rgb}{0,0,0.75}
\hypersetup{
    colorlinks, 
    linkcolor={my-linkcolor},
    citecolor={my-citecolor}, 
    urlcolor={my-urlcolor}
}

\DeclareMathOperator{\bbC}{\mathbb{C}}
\DeclareMathOperator{\bbR}{\mathbb{R}}

\DeclareMathOperator{\frakL}{\mathfrak{L}}

\DeclareMathOperator{\cp}{\mathbb{C}P}
\DeclareMathOperator{\ic}{IC}
\DeclareMathOperator{\pr}{\operatorname{pr}}
\DeclareMathOperator{\too}{\longrightarrow}
\DeclareMathOperator{\surj}{\twoheadrightarrow}

\declaretheorem[parent=section]{theorem}
\declaretheorem[unnumbered, name=Theorem]{theorem*}
\declaretheorem[sibling=theorem]{proposition}
\declaretheorem[unnumbered, name=Proposition]{proposition*}

\declaretheorem[unnumbered, name=Conjecture]{conjecture*}

\declaretheorem[unnumbered, style=definition, name=Definition]{definition*}
\declaretheorem[sibling=theorem]{lemma}
\declaretheorem[unnumbered, name=Lemma]{lemma*}
\declaretheorem[sibling=theorem, style=remark]{remark}
\declaretheorem[unnumbered, style=remark, name=Remark]{remark*}
\declaretheorem[sibling=theorem, style=remark]{assumption}
\declaretheorem[unnumbered, style=remark, name=Assumption]{assumption*}

\usepackage{fullpage, microtype}

\title{Torus actions and tensor products of intersection cohomology}
\author{Asilata Bapat}
\date{\today}
\address{Department of Mathematics, The University of Chicago, Chicago, IL 60637}
\email{asilata@math.uchicago.edu}

\begin{document}
\begin{abstract}
  Given certain intersection cohomology sheaves on a projective variety with a torus action, we relate the cohomology groups of their tensor product to the cohomology groups of the individual sheaves. We also prove a similar result in the case of equivariant cohomology.
\end{abstract}
\maketitle

\section{Introduction}
\label{sec:intro}
Let $X$ be a smooth complex projective variety together with an action of an algebraic torus $T$ with isolated fixed points. We fix a one-parameter subgroup $\lambda\colon \bbC^*\to T$, such that the $\lambda$-fixed points on $X$ are exactly the $T$-fixed points. Let $W$ denote the (finite) set of fixed points. Consider the Bia\l{}ynicki-Birula decomposition (see, e.g., \cite{bialynicki-birula}) of $X$, defined as follows. For each $w\in W$ the \emph{attracting set} is
\[
X_w = \{x\in X\mid \lim_{t\to 0}\lambda(t)\cdot x = w\},\quad t\in \bbC^*.
\]
Then each $X_w$ is a $\lambda$-stable affine space, and hence the decomposition $X = \coprod_{w\in W}X_w$ is a cell decomposition. For the purposes of this paper, we make the following additional assumptions on the $T$-action on $X$.
\begin{assumption}
  \label{assum:stratification}
  The cell decomposition $X = \coprod_{w\in W}X_w$ is an algebraic stratification of $X$. In particular, the closure $\overline{X_w}$ of every cell $X_w$ is a union of cells.
\end{assumption}
\begin{assumption}
  \label{assum:contraction}
  For each $w\in W$, there is a one-parameter subgroup $\lambda_w\colon \bbC^*\to T$ and a neighbourhood $V_w$ of $w$ such that $\lim_{t\to 0}\lambda_w(t)\cdot v = w$ for every $v\in V_w$ and $t\in \bbC^*$.  
\end{assumption}

For each $w\in W$, let $\ic_w$ denote the intersection cohomology complex on the closure of the cell $X_w$, extended by zero to all of $X$. The main theorem of the paper describes the cohomology of the tensor products of a collection of $\ic_w$, in terms of the tensor products of the cohomologies of the individual $\ic_w$.

\subsection{Main result} Let $H = H^\bullet(X)$ be the cohomology ring of the base space $X$. For any complex $\frakL$ on $X$, its cohomology $H^\bullet(\frakL) = H^\bullet(X, \frakL)$ is a module over $H$. In particular, each $H^\bullet(\ic_w)$ is an $H$-module.


Let $\Delta\colon X\to X^m$ be the diagonal embedding. Consider any complexes $\frakL_1,\ldots,\frakL_m$ in the category $D^b(X)$, which is the bounded derived category of constructible sheaves on $X$. Then their derived tensor product is also an element of $D^b(X)$, and will be denoted by $\frakL_1\otimes\cdots\otimes\frakL_m$. Recall that
\[
\frakL_1\otimes\cdots\otimes \frakL_m = \Delta^*(\frakL_1\boxtimes\cdots\boxtimes\frakL_m).
\]
There is a natural cup-product map that connects the cohomology spaces $H^\bullet(\frakL_i)$ to the cohomology of the tensor product $\frakL_1\otimes\cdots\otimes \frakL_m$.
\begin{theorem}
  \label{thm:maintheorem}
  Let $(w_1,\ldots,w_m)$ be an $m$-tuple of $T$-fixed points of $X$. If the assumptions
  \ref{assum:stratification} and \ref{assum:contraction} hold, then the natural map
  \begin{equation}
  \label{eq:cohom-mult}
  H^\bullet(\ic_{w_1})\otimes_H\cdots\otimes_HH^\bullet(\ic_{w_m}) \to H^\bullet(\ic_{w_1}\otimes \cdots\otimes\ic_{w_m})
\end{equation}
  is an isomorphism.
\end{theorem}

Recall that since $X$ is a $T$-space, each IC sheaf $\ic_{w_j}$ is $T$-equivariant, and so is the tensor product $\ic_{w_1}\otimes\cdots\otimes \ic_{w_m}$. Let $H_T = H^\bullet_T(X)$ be the $T$-equivariant cohomology of $X$. For any $T$-equivariant complex $\frakL$ on $X$, its $T$-equivariant cohomology $H^\bullet_T(\frakL) = H^\bullet_T(X,\frakL)$ is a graded $H_T$-module. As before, there is a cup product map for $T$-equivariant cohomology, which factors through $H_T$.

\begin{theorem}
  \label{thm:maintheorem-equivariant}
  Under the assumptions \ref{assum:stratification} and \ref{assum:contraction}, the natural map
  \[
  H^\bullet_T(\ic_{w_1})\otimes_{H_T}\cdots \otimes_{H_T}H_T^\bullet(\ic_{w_m})\to H^\bullet_T(\ic_{w_1}\otimes\cdots\otimes \ic_{w_m})
  \]
  is an isomorphism.
\end{theorem}

\begin{remark}
  \autoref{thm:maintheorem-equivariant} is used in \cite{gk}.
\end{remark}

In \autoref{sec:setup} we recall some standard constructions and some notation. In \autoref{sec:mainproof} and \autoref{sec:mainproof-equivariant}, we describe the proofs of \autoref{thm:maintheorem} and \autoref{thm:maintheorem-equivariant} respectively.

\section{Setup}
\label{sec:setup}
We describe some additional setup and notation, following \cite{ginsburg}. Consider $S^1\subset \bbC^*$, and the $S^1$-action on $X$ under $\lambda$. There is a moment map $f\colon X\to \bbR$ for this action, which is known to be a Morse function with critical points exactly the fixed points of $\lambda$ (see, e.g., \cite{audin}). Moreover, each cell of the Morse decomposition of $X$ with respect to $f$ is a union of Bia\l{}ynicki-Birula cells.

Let $a_0 < a_1 < \cdots$ be the critical values of $f$. Let $X_n = f^{-1}((-\infty, a_n])$. Then $X_0 \subset X_1 \subset \cdots$ is an increasing filtration of $X$ by closed subvarieties. Set $U_n = X_n\backslash X_{n-1}$. We may assume that $U_n$ consists of a single stratum $X_w$; otherwise we can replace the filtration $\{X_n\}$ of $X$ by a suitable refinement. We have the following inclusions:
\[
X_n\stackrel{i_n}{\hookrightarrow} X, \quad X_{n-1}\stackrel{v}{\hookrightarrow} X_n \stackrel{u}{\hookleftarrow} U_n.
\]
Recall that the \emph{minus-decomposition} $X = \coprod_{w\in W} X_w^-$ is the cell decomposition given by the expanding sets of the fixed points $w$. Suppose that $w\in X_n\backslash X_{n-1}$ for some $n$. Then the closed submanifolds $\overline{X_w^-}$ and $X_n$ intersect transversally in the single point $\{w\}$. Let $c_n\in H^\bullet(X)$ be the Poincar\'e dual to the homology class of $\overline{X_w^-}$. The class $c_n$ may be interpreted as the Thom class of the normal bundle to $\overline{X_w^-}$ in $X$. The vector space $H^\bullet(X)$ is generated by the collection $\{c_n\}$. Finally let $L_{j,n} = i_n^*\ic_{w_j}$ for each $j$ and $n$.

\begin{proposition}
  \label{prop:mainprop}
  For every $n$, the multiplication map
  \begin{equation}
    \label{eq:mult-map-inductive}
    H^\bullet(L_{1,n})\otimes_H\cdots \otimes_HH^\bullet(L_{m,n})\to H^\bullet(L_{1,n}\otimes\cdots\otimes L_{m,n})
  \end{equation}
  is an isomorphism.
\end{proposition}
If $n = \dim X$, then $L_{j,n} = \ic_{w_j}$ for each $j$. Hence \autoref{thm:maintheorem} follows from this proposition, and we now focus on proving the proposition.

\section{Proof of the isomorphism}
\label{sec:mainproof}
We prove \autoref{prop:mainprop} by induction on the dimension of the strata. In the base case of $n = 0$, the space $X_0$ is a single point, and hence each $L_{j,0}$ is just a one-dimensional vector space. Hence the map \eqref{eq:mult-map-inductive} reduces to the multiplication map $\bbC\otimes\cdots\otimes \bbC \to \bbC$, which is an isomorphism.

To prove the induction step, we mainly use the following distinguished triangles:
\begin{align}
  \label{eq:maintriangle}
  u_!u^*L_{j,n}&\to L_{j,n}\to v_*v^*L_{j,n},\\
  \label{eq:othertriangle}
  v_!v^!L_{j,n}&\to L_{j,n}\to u_*u^*L_{j,n}.
\end{align}
After taking cohomology, each of the above distinguished triangles produces a long exact sequence. In our case all connecting homomorphisms of this long exact sequence vanish (see, e.g. \cite[Lemma 20]{soergel} and \cite[Proposition 3.2]{ginsburg}).

For brevity, we will use the following notation through the remainder of the paper.
\[
  \begin{aligned}
    M_{m,n} &= L_{2,n}\otimes \cdots\otimes L_{m,n},\\
    A_{m,n} &= H^\bullet(L_{2,n})\otimes_H\cdots\otimes_H H^\bullet(L_{m,n}),\\
    B_{m,n} &= H^\bullet(u_*u^*L_{2,n})\otimes_H\cdots\otimes_H H^\bullet(u_*u^*L_{m,n}).
  \end{aligned}
\]

\begin{lemma}\label{lem:purity-ses}\mbox{}
  \begin{enumerate}[(i)]
  \item The cohomology $H^\bullet(L_{1,n}\otimes\cdots\otimes L_{m,n})$ is pure. \label{item:purity}
  \item There is a short exact sequence
    \[
    0 \to H^\bullet_c(u^*L_{1,n}\otimes \cdots \otimes u^*L_{m,n})\to H^\bullet(L_{1,n}\otimes\cdots\otimes L_{m,n})\to H^\bullet(L_{1,n-1}\otimes\cdots\otimes L_{m,n-1})\to 0.
    \]
    \label{item:purity-ses}
  \end{enumerate}
\end{lemma}
\begin{proof}
The proof is by induction on $n$. When $n = 0$, we have $X_{-1} = \emptyset$ and $U = X_0$. The open inclusion $u$ is the zero map, and the closed inclusion $v$ is the identity map, hence (\ref{item:purity-ses}) is clear. In this case, all sheaves $L_{i,0}$ are supported on a single point, and hence $L_{1,0}\otimes\cdots\otimes L_{m,0}$ is also supported on a single point. In this situation, one can show (see, e.g. \cite{springer}) that $H^\bullet(L_{1,0}\otimes\cdots\otimes L_{m,0})$ is pure, which proves (\ref{item:purity}). Similar arguments have been used in \cite{ginsburg} and \cite{braden}.

For the induction step, consider the distinguished triangle \eqref{eq:maintriangle} for $L_{1,n}$, and apply the functor $(-\otimes L_{2,n}\otimes\cdots \otimes L_{m,n})$, which may be written as $(-\otimes M_{m,n})$ in our previously introduced notation. This yields the following distinguished triangle:
\[
u_!u^*L_{1,n}\otimes M_{m,n}\to L_{1,n}\otimes M_{m,n}\to v_*v^*L_{1,n}\otimes M_{m,n}.
\]
By a repeated application of the projection formula, we may write the first term of this triangle as 
\[
u_!u^*L_{1,n}\otimes M_{m,n} = u_!u^*L_{1,n}\otimes L_{2,n}\otimes\cdots \otimes L_{m,n}\cong u_!(u^*L_{1,n}\otimes\cdots\otimes u^*L_{m,n}),
\]
and the third term of this triangle as 
\begin{align*}
v_*v^*L_{1,n}\otimes M_{m,n} &= v_*v^*L_{1,n}\otimes L_{2,n}\otimes \cdots\otimes L_{m,n}\\
&\cong v_*(v^*L_{1,n}\otimes\cdots \otimes v^*L_{m,n})\\
&= v_*(L_{1,n-1}\otimes \cdots L_{m,n-1}) = v_*(L_{1,n-1}\otimes M_{m,n-1}).
\end{align*}
Taking cohomology, we obtain the following long exact sequence.
\[
\cdots \to H^\bullet_c(u^*L_{1,n}\otimes\cdots \otimes u^*L_{m,n})\to H^\bullet(L_{1,n}\otimes M_{m,n})\to H^\bullet(L_{1,n-1}\otimes M_{m,n-1})\to \cdots.
\]
The term $H^\bullet(L_{1,n-1}\otimes M_{m,n-1})$ is pure by the induction hypothesis. Recall that  $u^*L_{1,n}\otimes \cdots \otimes u^*L_{m,n}$ is a direct sum of shifted constant sheaves. To prove purity of $H^\bullet(L_{1,n-1}\otimes M_{m,n-1})$, we use the following lemma, which is a consequence of \cite[(7.24)]{schmid}.
\begin{lemma}
  Let $M$ be a simply-connected subset of a smooth projective variety and let $\underline{M}$ be the constant sheaf on $M$, together with a mixed Hodge structure. Let $x$ be any point of $M$, and let $j_x\colon \{x\}\to M$ be the inclusion. Then $H^\bullet(\underline{M})$ is pure if and only if $H^\bullet(j_x^*\underline{M})$ is pure.
\end{lemma}
Hence the cohomology  $H^\bullet_c(u^*L_{1,n}\otimes\cdots \otimes u^*L_{m,n})$ is pure. Since the terms on either side are pure, the connecting homomorphisms of the long exact sequence vanish, and $H^\bullet(L_{1,n}\otimes M_{m,n})$ is also pure. This argument completes the induction step.
\end{proof}

\begin{lemma}
  \label{lem:condensed-les}
  There is an exact sequence
  \[
  H^\bullet(u_!u^*L_{1,n})\otimes_H B_{m,n}\to H^\bullet(L_{1,n})\otimes_H A_{m,n}\to H^\bullet(v_*v^*L_{1,n})\otimes_H A_{m,n}\to 0.
  \]
\end{lemma}

\begin{proof}
Consider the distinguished triangle \eqref{eq:maintriangle} for the sheaf $L_{1,n}$. Taking cohomology and applying the functor $-\otimes_HA_{m,n}$, we obtain the right-exact sequence
\[
H^\bullet(u_!u^*L_{1,n})\otimes_H A_{m,n}\stackrel{f}{\too} H^\bullet(L_{1,n})\otimes_HA_{m,n} \stackrel{g}{\too} H^\bullet(v_*v^*L_{1,n})\otimes_H A_{m,n}\to 0.
\]
Using the distinguished triangles \eqref{eq:othertriangle} for each of the sheaves $L_{j,n}$ for $j\geq 2$, we have surjective morphisms
\[
H^\bullet(L_{j,n})\surj H^\bullet(u_*u^*L_{j,n}).
\]
Taking the tensor product of all of these along with $H^\bullet(u_!u^*L_{1,n})$, we obtain a surjective morphism
\[
H^\bullet(u_!u^*L_{1,n})\otimes_HA_{m,n}\stackrel{h}{\surj} H^\bullet(u_!u^*L_{1,n})\otimes_H B_{m,n}.
\]

We now show that the map $f$ factors through the map $h$, by showing that $f(\ker h) = 0$. Since all boundary maps in the cohomology long exact sequence of the triangles \eqref{eq:othertriangle} vanish, the following set generates $\ker h$:
\[
\{a_1\otimes a_2\otimes \cdots \otimes a_n\mid a_j\in H^\bullet(v_*v^!L_{j,n})\text{ for some } 2\leq j\leq m \}.
\]
Consider any element $a_1\otimes a_2\otimes\cdots \otimes a_n\in \ker h$. Suppose that $a_j\in H^\bullet(v_*v^!L_{j,n})$. From \cite{ginsburg}, we know that $c_na_j = 0$, and that $a_1 \in c_nH^\bullet(L_{1,n})$. Since all tensor products are over $H$, the image of $h(a_1\otimes\cdots\otimes a_n)$ under $f$ must be zero. Therefore $f$ factors through $h$, and we obtain the desired short exact sequence.
\end{proof}

Set $M_{m,n}= L_{2,n}\otimes\cdots \otimes L_{m,n}$. Putting together the exact sequences from \autoref{lem:condensed-les} and \autoref{lem:purity-ses}, we obtain the following commutative diagram, where the vertical maps are cup product maps.
\begin{equation}
  \label{eq:big-comm-diagram}
  \begin{gathered}
    \begin{tikzpicture}[column sep=2em, row sep=2em]
      \matrix(m)[matrix of math nodes]
      {&H^\bullet(u_!u^*L_{1,n})\otimes_H B_{m,n}& H^\bullet(L_{1,n})\otimes_H A_{m,n} & H^\bullet(v_*v^*L_{1,n})\otimes_H A_{m,n}&0\\
        0&H^\bullet(u_!u^*L_{1,n}\otimes M_{m,n})& H^\bullet(L_{1,n}\otimes M_{m,n})& H^\bullet(v_*v^*L_{1,n}\otimes M_{m,n})&0\\};
      \path[->, font=\scriptsize]
      (m-1-2) edge (m-1-3) edge node[auto] {$a$} (m-2-2)
      (m-1-3) edge node[auto] {$b$} (m-2-3) edge (m-1-4)
      (m-1-4) edge (m-1-5) edge node[left] {$c$} (m-2-4)
      (m-2-1) edge (m-2-2)
      (m-2-2) edge (m-2-3)
      (m-2-3) edge (m-2-4)
      (m-2-4) edge (m-2-5);
    \end{tikzpicture}
  \end{gathered}
\end{equation}

The following two lemmas prove that the map $a$ is an isomorphism.
\begin{lemma}
  \label{lem:mapa-twofactors}
  Let $F$ be any sheaf supported on $U = X_n\backslash X_{n-1}$. Let $M$ be the restriction to $X_n$ of some IC sheaf $\ic_x$. Then the multiplication map
  \[
  H^\bullet(u_!F)\otimes_H H^\bullet(u_*u^*M)\to H^\bullet(u_!F\otimes u_*u^*M)
  \]
  is an isomorphism.
\end{lemma}
\begin{proof}
    Consider the following commutative diagram, where $\sigma$ and $\pi$ are projections to a point.
  \begin{center}
    \begin{tikzpicture}[row sep=2em, column sep=2em]
      \matrix(m)[matrix of math nodes]
      {U & X_n\\
        & \operatorname{pt}\\};
      \path[->, font=\scriptsize]
      (m-1-1) edge node[above] {$u$} (m-1-2) edge node[left] {$\sigma$} (m-2-2)
      (m-1-2) edge node[auto] {$\pi$} (m-2-2);
    \end{tikzpicture}
  \end{center}
  We may write the following:
  \begin{align*}
    H^\bullet(u_!F) &= \pi_*(u_!F) = \sigma_!F,\\
    H^\bullet(u_*u^*M) &= \pi_*(u_*u^*M) = \sigma_*u^*M.
  \end{align*}
  We now use the projection formula for the maps $\pi$ and $\sigma$. The projection formula isomorphism is functorial and commutes with compositions of maps. Hence we obtain the following commutative diagram.
  \begin{equation}
    \label{eq:adj-chase-diagram}
    \begin{gathered}
      \begin{tikzpicture}[row sep=2em, column sep=3em]
        \matrix(m)[matrix of math nodes]
        {\pi_*(u_!F)\otimes \pi_*(u_*u^*M)& \sigma_!(F)\otimes \sigma_*(u^*M)\\
          \pi_*(u_!F\otimes \pi^*\pi_*u_*u^*M)& \sigma_!(F\otimes \sigma^*\sigma_*(u^*M))\\
          \pi_*(u_!F\otimes u_*u^*M)& \sigma_!(F\otimes u^*u_*u^*M)\\};
        \path[->, font=\scriptsize]
        (m-1-1) edge node[above] {$=$} (m-1-2) edge node[left] {proj.} node[right] {$\cong$} (m-2-1)
        (m-1-2) edge node[left] {proj.} node[right] {$\cong$} (m-2-2)
        (m-2-1) edge node[above] {$\cong$} node[below] {proj.} (m-2-2) edge node[left] {adj.} (m-3-1)
        (m-2-2) edge node[left] {adj.} node[right] {$\cong$} (m-3-2)
        (m-3-1) edge node[below] {proj.} node[above] {$\cong$} (m-3-2);        
      \end{tikzpicture}
    \end{gathered}
  \end{equation}
  The maps marked as `proj.' are all isomorphisms arising from applications of the projection formula. The maps marked as `adj.' are obtained from the counit map of the adjunction between $\pi^*$ and $\pi_*$. The right hand vertical map marked as `adj.' can be rewritten as follows:
  \[
  \sigma_!(F\otimes u^*\pi^*\pi_*u_*u^*M)\to \sigma_!(F\otimes u^*u_*u^*M).
  \]
  This map is an isomorphism because the map $u^*\pi^*\pi_*u_*u^*M\to u^*u_*u^*M$ is an isomorphism. Hence the composition of the right hand vertical maps is an isomorphism. From the diagram it is clear that the composition of the left hand vertical maps is an isomorphism as well. Since this map is simply the cup product map on cohomology, the lemma is proved.

\end{proof}

\begin{lemma}
  \label{lem:mapa-isomorphism}
  The multiplication map
  \[
  a\colon H^\bullet(u_!u^*L_{1,n})\otimes_H B_{m,n}\to H^\bullet(u_!u^*L_{1,n}\otimes M_{m,n})
  \]
  is an isomorphism.
\end{lemma}
\begin{proof}[Proof of lemma]
  We prove this by induction on $m$. The case of $m = 2$ is covered in \autoref{lem:mapa-twofactors}. Now suppose that 
  \[
  H^\bullet(u_!u^*L_{1,n})\otimes_HB_{m-1,n}\stackrel{\cong}{\too} H^\bullet(u_!u^*L_{1,n}\otimes M_{m-1,n})
  \]
  via the multiplication map. Tensoring with $H^\bullet(u_*u^*L_{m,n})$, we see that
  \begin{equation}
    \label{eq:induction-map1}
    H^\bullet(u_!u^*L_{1,n})\otimes_HB_{m,n}\stackrel{\cong}{\too} H^\bullet(u_!u^*L_{1,n}\otimes M_{m-1,n})\otimes_HH^\bullet(u_*u^*L_{m,n}).
  \end{equation}
  By the projection formula, $\left(u_!u^*L_{1,n}\otimes M_{m-1,n}\right)$ is isomorphic to $u_!(u^*L_{1,n}\otimes u^*M_{m-1,n})$, and hence it is supported on $U$. Now by \autoref{lem:mapa-twofactors}, we see that 
  \begin{equation}
    \label{eq:induction-map2}
    H^\bullet(u_!u^*L_{1,n}\otimes M_{m-1,n})\otimes_HH^\bullet(u_*u^*L_{m,n}) \stackrel{\cong}{\too} H^\bullet(u_!u^*L_{1,n}\otimes M_{m,n}).
  \end{equation}
  By composing \eqref{eq:induction-map1} and \eqref{eq:induction-map2}, we see that
  \[
  H^\bullet(u_!u^*L_{1,n})\otimes_H B_{m,n}\stackrel{\cong}{\too} H^\bullet(u_!u^*L_{1,n}\otimes M_{m,n})
  \]
  via the multiplication map $a$.
\end{proof}

The next lemma uses the induction hypothesis to prove that the map $c$ is an isomorphism.
\begin{lemma}
  \label{lem:mapc-isomorphism}
  The multiplication map
  \[
  c\colon H^\bullet(v_*v^*L_{1,n})\otimes_HH^\bullet(M_{m,n})\to H^\bullet(v_*v^*L_{1,n}\otimes M_{m,n})
  \]
\end{lemma}
\begin{proof}[Proof of lemma]
First observe that $v^*L_{1,n} = L_{1,n-1}$ and $v^*M_{m,n} = M_{m,n-1}$. Using the projection formula, we have
\[
v_*v^*L_{1,n}\otimes M_{m,n} \cong v_*(L_{1,n-1}\otimes v^*M_{m,n}) = v_*(L_{1,n-1}\otimes M_{m,n-1}).
\]
Next, the element $c_n\in H$ acts on $H^\bullet(v_*L_{1,n-1})$ by zero, since $L_{1,n-1}$ is supported on $X_{n-1}$. Recall from \cite{ginsburg} that the cokernel of $c_n$ on $H^\bullet(M_{m,n})$ is just $H^\bullet(M_{m,n-1})$. Hence
\[
H^\bullet(v_*v^*L_{1,n})\otimes_H H^\bullet(M_{m,n})\cong H^\bullet(L_{1,n-1})\otimes_H H^\bullet(M_{m,n-1}).
\]
Hence the map $c$ can be rewritten as the multiplication map
\[
H^\bullet(L_{1,n-1})\otimes_HH^\bullet(M_{m,n-1})\to H^\bullet(L_{1,n-1}\otimes M_{m,n-1}),
\]
which is an isomorphism by the induction hypothesis.
\end{proof}

Since the maps $a$ and $c$ from \eqref{eq:big-comm-diagram} are isomorphisms, we conclude by the snake lemma that the middle map $b$ is an isomorphism as well. Hence the induction step is proved.

\section{Computation of equivariant cohomology}
\label{sec:mainproof-equivariant}
Consider a smooth complex projective variety $X$ with the same assumptions as before. The goal of this section is to prove \autoref{thm:maintheorem-equivariant}. 

First we recall some constructions in equivariant cohomology. The main references are \cite{bernstein-lunts} and \cite{gkm}. Since each stratum $X_w$ is a locally closed $T$-invariant affine subset of $X$, the trivial local system on $X_w$ gives rise to the $T$-equivariant IC sheaf $\ic_w$ (see, e.g., \cite[Section 5.2]{bernstein-lunts}).
Consider the following diagram, where the map $\pr$ is the second projection, and the map $q$ is the quotient by the diagonal $T$-action.
\begin{center}
  \begin{tikzpicture}[row sep=2em, column sep=2em]
    \matrix(m)[matrix of math nodes]
    {
      &ET \times X&\\
      X&&ET\times_T X\\
    };
    \path[->, font=\scriptsize]
    (m-1-2) edge node[above left] {pr} (m-2-1) edge node[above right] {$q$} (m-2-3);
  \end{tikzpicture}
\end{center}
More precisely, the $T$-equivariant IC sheaf may be described as a triple $(\ic_w, \overline{\ic_w}, \beta)$, where $\ic_w$ is the IC sheaf on $X$ corresponding to $X_w$ and $\overline{\ic_w}$ is a canonically-defined sheaf on $ET\times_T X$, together with an isomorphism $\beta\colon \pr^*\ic_w\stackrel{\cong}{\too} q^*\overline{\ic_w}$ (see, e.g., \cite{bernstein-lunts}).

The equivariant cohomology ring of $X$ is denoted $H_T = H^\bullet_T(X)$. In our case, $H_T$ is isomorphic to to $H^\bullet(X)\otimes H^\bullet(BT)$ (see, e.g., \cite[Theorem 14.1]{gkm}). For any $T$-equivariant sheaf on $X$, its $T$-equivariant cohomology is a module over $H_T$.

If $Y$ is a closed $T$-equivariant submanifold of $X$, we can construct its equivariant Poincar\'e dual as follows. The normal bundle $E_Y$ of $Y$ is a $T$-equivariant vector bundle. Hence there is an equivariant Thom class $\tau_Y\in H^\bullet_{T,c}(E_Y)$. By the equivariant tubular neighborhood theorem, $E_Y$ can be equivariantly embedded inside $X$. Then the extension by zero of $\tau_Y$ under this embedding is a class in $H^\bullet_T(X)$, which is called the equivariant Poincar\'e dual of $Y$.

Since the submanifolds $\overline{X_w^-}$ are closed and $T$-equivariant, we can construct their equivariant Poincar\'e duals. If $w$ is the unique $T$-fixed point in $X_n\backslash X_{n-1}$, let $\widetilde{c}_n\in H^\bullet_T(X)$ denote the equivariant Poincar\'e dual to $\overline{X_w^-}$. Each class $\widetilde{c}_n$ restricts to the previously defined class $c_n$ under the map $H_T\to H^\bullet(X)$, hence the collection $\{\widetilde{c}_n\}$ generates $H_T$ over $H^\bullet(BT)$.

Recall that $X_n$ intersects $\overline{X_w^-}$ transversally in the point $w$, and that $\widetilde{c}_n$ is the Thom class of the normal bundle of $\overline{X_w^-}$. Hence the restriction of $\widetilde{c}_n$ to $X_n$ is the image in $H_T^\bullet(X_n)$ of a generator of the local cohomology group $H^\bullet_T(X_n, X_n\backslash \{w\})$. Set $U_n = X_n \backslash X_{n-1}$. Since $w\in U_n$, we have  $H^{\bullet}_T(X_n, X_n\backslash\{w\}) \cong H^{\bullet}_T(U_n, U_n\backslash\{w\})$ by excision. But $U_n$ is an affine space that is $T$-equivariantly contractible to $w$, and hence $H^{\bullet}_T(U_n, U_n\backslash\{w\}) \cong H^\bullet_{T,c}(U_n)$. Moreover, multiplication by $\widetilde{c}_n$ annihilates the cohomology of any sheaf supported on $X_{n-1}$. Hence the map of multiplying by $\widetilde{c}_n$ on $H^\bullet_T(X_n)$ factors through $H^\bullet_T(U_n)$, and sends its generator to a generator of $H^\bullet_{T,c}(U_n)\subset H_T^\bullet(X_n)$. an isomorphism $H^\bullet_T(U_n)\cong H^\bullet_{T,c}(U_n)$.

Now let $L$ be any complex on $X_n$ that breaks up as a direct sum of constant sheaves when restricted to $U_n$. Let $u\colon U_n\to X_n$ be the inclusion. By the same argument, we obtain the following commutative diagram, which is an analogue of \cite[3.8a]{ginsburg}:
\begin{center}
  \begin{tikzpicture}[row sep=2em, column sep=2em]
    \matrix(m)[matrix of math nodes]
    {
      H^\bullet_T(L)& H^\bullet_T(u^*L)\\
      H^\bullet_T(L)& H^\bullet_{T,c}(u^*L)\\
    };
    \path[->, font=\scriptsize]
    (m-1-1) edge[->>] (m-1-2) edge node[left] {$\widetilde{c}_n$} (m-2-1)
    (m-1-2) edge node[left] {$\widetilde{c}_n$} node[right] {$\cong$}(m-2-2)
    (m-2-2) edge[left hook->] (m-2-1);
  \end{tikzpicture}
\end{center}

Now consider an $m$-tuple $(w_1,\ldots,w_m)$ of $T$-fixed points of $X$. Then $\ic_{w_1},\ldots,\ic_{w_m}$ are the IC sheaves on $X_{w_1},\ldots,X_{w_m}$ respectively. Let $L_{j,n} = i_n^*\ic_{w_j}$ for each $j$ and $n$. 
\begin{proposition}
  Under the assumptions \ref{assum:stratification} and \ref{assum:contraction}, the natural maps
  \[
  H^\bullet_T(L_{1,n})\otimes_{H_T}\cdots \otimes_{H_T}H_T^\bullet(L_{m,n})\to H^\bullet_T(L_{1,n}\otimes\cdots\otimes L_{m,n})
  \]
  are isomorphisms for each $n$.
\end{proposition}
If $n = \dim X$, then $L_{j,n} = \ic_{w_j}$ for each $j$. Hence this proposition implies \autoref{thm:maintheorem-equivariant}. To prove the proposition, we first state two general lemmas about $T$-equivariant cohomology of sheaves.
\begin{lemma}
  Consider the fiber bundle $ET\times_TX\to BT$, with fiber $X$. Let $\ic_w$ be the ($T$-equivariant) IC sheaf on the closure of a stratum $X_w$, extended by zero to all of $X$. Then the Leray spectral sequence for the computation of $H^\bullet_T(X;\ic_w) = H^\bullet(ET\times_TX;\overline{\ic_w})$ collapses at the $E_2$-page. Hence $H_T^\bullet(\ic_w)$ is isomorphic to $H^\bullet(\ic_w)\otimes H^\bullet(BT)$ as a graded $H_T$-module.
\end{lemma}
\begin{proof}
  See \cite[Theorem 14.1]{gkm}. The proof uses the fact that the cohomology of $BT\cong (\cp^\infty)^{\dim T}$ is pure.
\end{proof}

\begin{lemma}
  \label{lem:equivariant-purity}
  Let $Y$ be any $T$-space, and let $\frakL$ be a $T$-equivariant sheaf on $Y$ such that the space $H^\bullet(Y;\frakL)$ is pure. Then $H^\bullet_T(Y;\frakL)$ is pure as well.
\end{lemma}
\begin{proof}
  Recall that $H^\bullet_T(Y,\frakL) = H^\bullet(ET\times_TX, \overline{\frakL})$. The result follows from computing the Leray spectral sequence  for the fiber bundle $ET\times_T Y\to BT$, and by using that $H^\bullet(BT)$ and $H^\bullet(Y,\frakL)$ are pure.
\end{proof}

We also record several results that are equivariant analogues of results from \autoref{sec:mainproof}.
\begin{lemma}\mbox{}
  \label{lem:equivariant-analogues}
  \begin{enumerate}[(i)]
  \item The boundary maps in the long exact sequences of $T$-equivariant cohomology for the distinguished triangles \eqref{eq:maintriangle} and \eqref{eq:othertriangle} vanish. 
  \item The cohomology $H^\bullet_T(L_{1,n}\otimes\cdots\otimes L_{m,n})$ is pure.
  \item There is a short exact sequence
    \[
    0\to H^\bullet_{T,c}(u^*L_{1,n}\otimes\cdots\otimes u^*L_{m,n})\to H_T^\bullet(L_{1,n}\otimes\cdots\otimes L_{m,n})\to H^\bullet_T(L_{1,n-1}\otimes\cdots\otimes L_{m,n-1})\to 0.
    \] 
  \end{enumerate}

\end{lemma}
\begin{proof}
  The proofs are analogous to the proofs of their counterparts from \autoref{sec:mainproof}, using the observation of \autoref{lem:equivariant-purity} and the fact that $H^\bullet(BT)$ is pure.
\end{proof}

Recall that $M_{m,n}$ denotes $L_{2,n}\otimes \cdots \otimes L_{m,n}$. For brevity, we set up the following additional notation.
\begin{align*}
  \overline{A}_{m,n}& = H^\bullet_T(L_{2,n})\otimes_{H_T}\cdots \otimes_{H_T}H^\bullet_T(L_{m,n}),\\
  \overline{B}_{m,n}& = H^\bullet_T(u_*u^*L_{2,n})\otimes_{H_T}\cdots \otimes_{H_T}H^\bullet_T(u_*u^*L_{m,n}).
\end{align*}

\begin{lemma}
  \label{lem:condensed-les-equivariant}
  There is an exact sequence
  \[
  H^\bullet_T(u_!u^*L_{1,n})\otimes_{H_T}\overline{B}_{m,n}\to H^\bullet_T(L_{1,n})\otimes_{H_T}\overline{A}_{m,n}\to H^\bullet_T(v_*v^*L_{1,n})\otimes_{H_T}\overline{A}_{m,n}\to 0.
  \]
\end{lemma}

\begin{proof}
  The proof is analogous to the proof of \autoref{lem:condensed-les}, using the fact that $H^\bullet_T(X)\cong H^\bullet(X)\otimes H^\bullet(BT)$.
\end{proof}

\begin{lemma}
  \label{lem:mapa-twofactors-equivariant}
  Let $F$ be any $T$-equivariant sheaf supported on $U = X_n\backslash X_{n-1}$. Let $M$ be the restriction to $X_n$ of some IC sheaf $\ic_x$. Then the multiplication map
  \[
  H^\bullet_T(u_!F)\otimes_{H^\bullet(BT)}H_T^\bullet(u_*u^*M)\to H_T^\bullet(u_!F\otimes u_*u^*M)
  \]
  is an isomorphism.
\end{lemma}
\begin{proof}
  Consider the fiber bundle $ET\times_TX_n\to BT$, with fiber $X_n$. The $E_2$ pages of the Leray spectral sequences for $u_!F$ and $u_*u^*M$ are as follows:
  \begin{align*}
    H^p(BT,H^q(u_!F)) &\implies H^{p+q}_T(u_!F),\\
    H^r(BT,H^s(u_*u^*M)) &\implies H^{r+s}_T(u_*u^*M).
  \end{align*}
  On the $E_2$ page, the multiplication map can be written as the composition of the following two maps. The first map is the cup product with local coefficients, and the second is the fiber-wise cup product on the local systems.
  \begin{align*}
    H^p(BT, H^q(u_!F))\otimes_{H^\bullet(BT)}H^r(BT,H^s(u_*u^*M))&\to H^{p+r}(BT, H^q(u_!F)\otimes H^r(u_*u^*M)),\\
    H^{p+r}(BT, H^q(u_!F)\otimes H^s(u_*u^*M)) &\to H^{p+r}(BT, H^{q+s}(u_!F\otimes u_*u^*M)).
  \end{align*}
  Since the local systems $H^q(u_!F)$ and $H^s(u^*u_*M)$ are constant on $BT$, the first map yields isomorphisms
  \[
  H^\bullet(BT, H^q(u_!F))\otimes_{H^\bullet(BT)}H^\bullet(BT, H^s(u^*u_*M)) \stackrel{\cong}{\too} H^\bullet(BT, H^q(u_!F)\otimes H^s(u^*u_*M)).
  \]
  Finally, we know from \autoref{lem:mapa-twofactors} that $H^\bullet(u_!F)\otimes H^\bullet(u_*u^*M)\stackrel{\cong}{\too} H^\bullet(u_!F\otimes u_*u^*M)$ via the multiplication map. Altogether, the multiplication maps on the $E_2$ page yield an isomorphism 
  \[
  H^\bullet(BT,H^\bullet(u_!F))\otimes_{H^\bullet(BT)}H^\bullet(BT,H^\bullet(u_*u^*M))\stackrel{\cong}{\longrightarrow} H^\bullet(BT,H^\bullet(u_!F\otimes u_*u^*M)).
  \]

  The left hand side is a tensor product of two free $H^\bullet(BT)$-modules over $H^\bullet(BT)$. Hence it converges to $H^\bullet_T(u_!F)\otimes_{H^\bullet(BT)}H^\bullet_T(u_*u^*M)$. The right hand side converges to $H^\bullet_T(u_!F\otimes u_*u^*M)$. Since the $E_2$ pages of the left hand side and the right hand side are isomorphic via the multiplication map, the following multiplication map 
  \[
  H^\bullet_T(u_!F)\otimes_{H^\bullet(BT)}H^\bullet_T(u_*u^*M)\to H^\bullet_T(u_!F\otimes u_*u^*M) 
  \]
  is an isomorphism.
\end{proof}

\begin{proof}[Proof of \autoref{thm:maintheorem-equivariant}]
  We obtain the following commutative diagram from the exact sequences of \autoref{lem:equivariant-analogues} and \autoref{lem:condensed-les-equivariant}.
    \begin{equation}
  \label{eq:big-comm-diagram-equivariant}
  \begin{gathered}
    \begin{tikzpicture}[column sep=2em, row sep=2em]
      \matrix(m)[matrix of math nodes]
      {&H_T^\bullet(u_!u^*L_{1,n})\otimes_{H_T} \overline{B}_{m,n}& H_T^\bullet(L_{1,n})\otimes_{H_T} \overline{A}_{m,n} & H_T^\bullet(v_*v^*L_{1,n})\otimes_{H_T} \overline{A}_{m,n}&0\\
        0&H_T^\bullet(u_!u^*L_{1,n}\otimes M_{m,n})& H_T^\bullet(L_{1,n}\otimes M_{m,n})& H_T^\bullet(v_*v^*L_{1,n}\otimes M_{m,n})&0\\};
      \path[->, font=\scriptsize]
      (m-1-2) edge (m-1-3) edge node[auto] {$a$} (m-2-2)
      (m-1-3) edge node[auto] {$b$} (m-2-3) edge (m-1-4)
      (m-1-4) edge (m-1-5) edge node[left] {$c$} (m-2-4)
      (m-2-1) edge (m-2-2)
      (m-2-2) edge (m-2-3)
      (m-2-3) edge (m-2-4)
      (m-2-4) edge (m-2-5);
    \end{tikzpicture}
  \end{gathered}
\end{equation}

Now we can prove by induction on $m$ that the map $a$ is an isomorphism. First observe that the action of $H_T$ on $H^\bullet_T(u_!u^*L_{1,n})$ and on $\overline{B}_{m,n}$ factors through the map $H_T\to H^\bullet_T(U)\cong H^\bullet(BT)$, so
\[
H^\bullet_T(u_!u^*L_{1,n})\otimes_{H_T}\overline{B}_{m,n}\cong H^\bullet_T(u_!u^*L_{1,n})\otimes_{H^\bullet(BT)}\overline{B}_{m,n}.
\]
Hence the case of $m = 2$ is covered by \autoref{lem:mapa-twofactors-equivariant}. The general case is proved using an argument similar to that in \autoref{lem:mapa-isomorphism}. An argument similar to the proof of \autoref{lem:mapc-isomorphism} proves that the map $c$ is an isomorphism.

Hence by the snake lemma, the middle map $b$ is an isomorphism as well. Consequently, we obtain the following isomorphisms for every $n$:
\[
H^\bullet_T(L_{1,n})\otimes_{H_T}\cdots \otimes_{H_T}H_T^\bullet(L_{m,n})\to H^\bullet_T(L_{1,n}\otimes\cdots\otimes L_{m,n}).
\]
Setting $n = \dim X$, we see that the natural map
\[
H^\bullet_T(\ic_{w_1})\otimes_{H_T}\cdots \otimes_{H_T}H_T^\bullet(\ic_{w_m})\to H^\bullet_T(\ic_{w_1}\otimes\cdots\otimes \ic_{w_m})
\]
is an isomorphism.
\end{proof}

\subsection*{Acknowledgments}
I would like to thank my advisor Victor Ginzburg for suggesting the problem and for his advice and guidance throughout the project.

\bibliographystyle{amsalpha}
\bibliography{bibliography}

\end{document}